\documentclass[12pt]{article}
\usepackage{amssymb,amscd,amsfonts,amsmath,graphicx,amsthm}
\usepackage{latexsym}
\usepackage{pgf,tikz}
\usetikzlibrary{arrows}
\usepackage[ansinew]{inputenc}
\usepackage{mathrsfs}
\usepackage{graphicx}
\usepackage{rotating}
\usepackage[usenames,dvipsnames]{pstricks}
\usepackage{epsfig}
\usepackage{pst-grad} 
\usepackage{pst-plot} 
\usepackage{fancyhdr}
\usepackage[top=3cm,bottom=2cm,left=3cm,right=2cm]{geometry} 
\usepackage{makeidx}
\usepackage{pstricks,pst-node,pst-plot,pst-tree,pstricks-add}

\newtheorem{theorem}{Theorem}
\newtheorem{lemma}[theorem]{Lemma}
\newtheorem{proposition}[theorem]{Proposition}
\newtheorem{definition}[theorem]{Definition}
\newtheorem{example}[theorem]{Example}
\newtheorem{corollary}[theorem]{Corollary}
\newtheorem{remark}[theorem]{Remark}

\begin{document}

\title{Intersecting families of extended balls in the Hamming spaces}
\author{Anderson N. Martinhão\footnote{The first author is supported by Capes} \ \footnote{E-mail: anderson.martinhao@gmail.com} and
Emerson L. Monte Carmelo\footnote{The second author is supported in part by CNPq.} \ \footnote{E-mail: elmcarmelo@uem.br}  \\
{\small Universidade Estadual de Maringá} \\
 {\small Departamento de Matemática} \\
 \date{\today}}
\maketitle

\begin{abstract}
A family $\mathcal{F}$ of subsets of a set $X$ is $t$-intersecting
if $\vert  A_i \cap A_j \vert \geq t$ for every $A_i, \; A_j \in
\mathcal{F}$. We study intersecting families in the Hamming
geometry. Given $X=\mathbb{F}_q^3$ a vector space over the finite
field $\mathbb{F}_q$, consider a family where each $A_i$ is an
extended ball, that is, $A_i$ is the union of all balls centered in
the scalar multiples of a vector. The geometric behavior of extended
balls is discussed. As the main result, we investigate a ``large"
arrangement of vectors whose extended balls are ``highly
intersecting". Consider the following covering problem: a subset
$\mathcal{H}$ of $\mathbb{F}_q^3$ is a short covering if the union
of the all extended balls centered in the elements of $\mathcal{H}$
is the whole space. As an application of this work, minimal
cardinality of a short covering is improved for some instances of $q$.\\

\noindent \textbf{MSC(2010):} 05D05, 05B40, 11T99. \\

\noindent \textbf{Keywords:} Intersecting family, ball, Hamming
distance, finite field, extremal pro\-blem.
\end{abstract}

\section{Introduction}\label{intro}

\subsection{Intersecting family}

A family $\mathcal{F}=\{A_1, \ldots, A_m \}$ of subsets of an
underlying set $X$ is \emph{$t$-intersecting} if $\vert  A_i \cap
A_j \vert \geq t$ for any $i \neq j$. A classical class of problems
in extremal combinatorics deals with the computation of the maximum
cardinality of a $t$-intersecting family under certain constraints.
Typically, the imposed conditions are $X=\{1, \ldots, n\}$ and
$\vert A_i \vert =k$ for any $1 \leq i \leq m$. The solution for
$t=1$ is called the Erd\H{o}s-Ko-Rado theorem \cite{EKR}. The
complete solution for arbitrary $t$ was proved by Katona \cite{K}.
Algebraic versions of the Erd\H{o}s-Ko-Rado theorem have been
investigated for intersecting chains of boolean algebra by Erd\H{o}s
et al. \cite{ESS}, and for subspaces of a finite vector space by
Czabarka  \cite{Cz}.

On the other hand, the characterization of the extremal families was
obtained by Ahlswede and Khachatrian \cite{AKh} in connection with
the diameter problem in Hamming spaces. Interplays between extremal
combinatorics and geometry in Hamming spaces present several
difficult problems (see \cite{AK, Ko}, for instance), some of them
are motivated by applications to information theory (see \cite{B,
Ho}). These contributions have investigated union or intersection of
suitable arrangements of balls and their relationships with lines,
hyper-planes, or other geometric configurations.

In this work, we investigate intersecting families under a new
perspective: each $A_i$ is an arrangement of balls, as described
below.

\subsection{Extended ball}
Let $X=\mathbb{F}_q^3$ be the vector space over the finite field
$\mathbb{F}_q$, where $q$ denotes a prime power. Recall that the
\emph{Hamming distance} between two vectors $u=(u_1, u_2, u_3)$ and
$v=(v_1, v_2, v_3)$ is  $d(u,v)= \vert \{ i \ : \ u_i \neq v_i \}
\vert$. The \emph{ball of center} $u$ and \emph{radius} $1$ is
denoted by $B(u)=\{ v \in \mathbb{F}_q^3 \ : \ d(u,v) \leq 1 \}$.

If each $A_i$ is a ball and $m=\vert \mathcal{F} \vert$ is
``sufficiently large", then clearly there are two disjoint balls.

Consider a variant induced by a geometric change: each center $u$ is
``replaced" by a line. More precisely, given a vector $u$ in
$\mathbb{F}_q^3$, the {\it extended ball} (along the line induced by
$u$) is defined as
\begin{equation}\label{e0}
E(u)=\bigcup_{\lambda \in \mathbb{F}_q}  B(\lambda u).
\end{equation}

The covering problem induced by extended balls in an arbitrary space
$\mathbb{F}_q^n$ is called {\it short covering}, motivated by the
fact that short covering might provide us a way to store non-linear
codes using less memory than the classical ones. Applications to the
classical numbers $K_{q}(n,R)$ (from covering codes) appear in
\cite{otavio} and some of its references. On theoretical viewpoints,
 results on short coverings have been obtained from distinct tools:
graph theory \cite{anderson}, linear algebra \cite{otavio}, ring
theory \cite{Nakaoka, YSBY}.

In this paper we investigate how the extended balls intersect one
another. However, a new obstacle arises here:
 the cardinality of $E(u)$ is not invariant, for example, $\vert E((0,0,0)) \vert= \vert
B((0,0,0)) \vert =3q-2$ but $\vert E((1,1,1)) \vert = 3q^2-2q$, see
\cite{carmelo}.

Consider the family $ \mathcal{G}=\{ E(u) : \; u \in \mathbb{F}_q^3
\}$. The ``bound"  $ \vert {E}(u) \cap {E}(v) \vert \geq \vert B(0)
\vert =3q-2$ shows us that $\mathcal{G}$ is a trivial $3q-2$
intersecting family. This bound can not be improved for all space,
because it is sharp for the case $u=(0,0,1)$ and $v=(1,2,0)$ in
$\mathbb{F}_5^3$, for instance.

\subsection{The main statement}

As an attempt to improve the trivial bound above, we investigate a
``large" subfamily such that each pair of extended balls has
intersection with ``bigger size".  More precisely, given a family
$\mathcal{F}$, let us introduce
\[
\theta(\mathcal{F})= \max \{ \; t\; : \;  \mathcal{F} \mbox { is
$t$-intersecting} \}.
\]
This parameter is closely related to concepts from extremal set
theory. A family $\mathcal{F}$ is a {\it weak $\Delta$-system} if
there is $\lambda$ such that $\vert A_i \cap A_j \vert =\lambda$ for
any $i \neq j$, introduced by Erd\H{o}s et.al \cite{EMR}. Note that
$\theta(\mathcal{F}) \geq \lambda$. The {\it intersection structure}
of a family $\mathcal{F}$  is the set
$$I(\mathcal{F})=\{ A_i \cap A_j \; : \; i \neq j \},$$
which was studied by Talbot\cite{Talbot}, for instance. The min-max
property $ \theta(\mathcal{F})= min\{ \vert C \vert \;: \; C \in
I(\mathcal{F})\}$ holds.

Take $\mathcal{H} = \{\widetilde{E}(u) \ : \ u \in \mathbb{F}_q^3,
\; u\neq(0,0,0) \}$ as an example. What about $\theta(\mathcal{H})$?
We will see that $\theta(\mathcal{H})=0$ as an immediate consequence
of Theorem \ref{bolas} . The behavior of these intersections is more
curious when restricted to the following environmental
\[
\mathcal{D}_q= \{ (u_1,u_2,u_3) \in \mathbb{F}_q^3 \ : \ u_1,u_2,u_3
\mbox{ are pairwise distinct and non-zero} \}.
\]
Indeed, this computation depends on the arithmetic form of $q$; more
precisely:

\begin{theorem}\label{dream} Given a prime power $q$, let $\mathcal{E} = \{ \widetilde{E}(u) \ : \ u \in \mathcal{D}_q \}$. Thus
\[
\theta(\mathcal{E})= \left\{
\begin{array}{ll}
2(q-1) & \mbox{ if } q-1 \not\equiv 0 \,(mod)\, 3\\
0      & \mbox{ if } q-1 \equiv 0 \, (mod)\, 3.
\end{array}
\right.
\]
\end{theorem}

As a consequence, Theorem \ref{dream} reveals a high degree of
intersection if $3$ does not divide $q-1$:  $ \vert {E}(u) \cap
{E}(v) \vert \geq \vert \widetilde{E}(u) \cap \widetilde{E}(v) \vert
+ \vert B(0) \vert \geq 5q+4$  for any $u$ and $v$ in
$\mathcal{D}_q$.

\subsection{An application to short coverings}

As a complement of this work, the impact of the previous results
into the short covering problem is discussed. The set
 $\mathcal{H} \subset \mathbb{F}_q^3$ is a \emph{short covering of $\mathbb{F}_q^3$} if
\begin{equation}\label{e1}
\bigcup_{h \in \mathcal{H}} E(h)= \mathbb{F}_q^3.
\end{equation}
What is the minimum number $c(q)$ of extended balls that cover the
whole space $\mathbb{F}_q^3$? The only known values are: $c(2)=1$,
$c(3)=3$, $c(4)=3$, and $c(5)=4$. The best known bounds for $q \geq
7$ are described below, according to \cite{anderson} and its
references.
\begin{equation}\label{bestinfsup}
\left\lceil \frac{q+1}{2} \right\rceil \leq c(q) \leq \left\{
\begin{array}{l}
(q+3)/2, \mbox{ if } q \equiv 3 \ (mod \ 4)\\
(q+5)/2, \mbox{ if } q \equiv 1 \ (mod \ 4)\\
3(q+4)/4, \mbox{ if } q \mbox{ is even}.
\end{array} \right.
\end{equation}
In particular, $4 \leq c(7) \leq 5$, $5 \leq c(8) \leq 9$, and $5
\leq c(9) \leq 7$.

\begin{theorem}\label{anderson}
The values $c(7)=5$, $c(8)=6$, and $c(9)=6$ hold.
\end{theorem}

This work is structured as follows. The geometry of the
substructures $ \widetilde{B}(u)={B}(u) \cap \mathcal{D}_q$ and $
\widetilde{E}(u)=E(u)\cap \mathcal{D}_q$ play a central role in our
research. Information on these sets and their cardinalities can be
derived from a group of transformations described in Section
\ref{Bq}. In contrast to the classical ball, the cardinality of
$\widetilde E(u)$ vanishes according to certain parameters, as
stated in Theorem \ref{bolas}.
 A harder problem is studied in Section \ref{Intersection},
namely, the intersection of extended balls restricted to the set
$\mathcal{D}_q$. Theorem \ref{dream} is proved in Section
\ref{prova}. The lower bounds on $c(q)$ are obtained in Section
\ref{lower bounds}. On the other hand, optimal upper bounds are
constructed in Section \ref{bounds}.

\section{Extended balls}  \label{Bq}

\subsection{Preliminaries: extended balls in $\mathbb{F}_q^3$}

 What kind of application preserves the cardinality of $E(u)$? In order to
 answer this question, we review briefly a well-known action on groups. We
recommend the book \cite{Ts} for further details.

Given a prime power $q$, $L_q$ denotes the group of non-singular
linear operators of $\mathbb{F}_q$. Let $L_q^{3}$ be the direct
product $L_q \oplus L_q \oplus L_q$. As usual, $S_3$ denotes the
symmetric group of degree $3$. A natural action of $S_3$ on the
group $L_q^{3}$ is obtained by permutation of coordinate. This
action induces the \emph{wreath product} of $\mathbb{F}^{*}_q$ by
$S_3$,
\[
S_{3} \ltimes L_q^{3} = \{ (\varphi,\sigma) \, : \, \varphi \in S_3
\mbox{ and } \sigma \in L_q^{3} \},
\]
The cardinality of $E(u)$ vanishes according to the weight of $u$.
Recall that the \emph{weight} of a vector $u=(u_1,u_2,u_3)$ denotes
the number $\omega(u)=\vert \{ i \ : \ u_i \neq 0 \}\vert $.

\begin{lemma}\cite{carmelo}\label{grazieli}
If the vectors $u$ and $v$ are in the same orbit of $\mathbb{F}^3_q$
by the action $S_3 \ltimes L_q^{3}$, then  the cardinalities of the
sets $E(u)$ and $E(v)$ are equal. In particular, $\vert E(u)\vert
=\vert E(v) \vert$ whenever  $\omega(u) =\omega(v)$.
\end{lemma}

\subsection{Extended balls in $\mathcal{D}_q$} This subsection provides
information on the sets:
\[ \widetilde{B}(u)={B}(u) \cap \mathcal{D}_q \;\;\; \mbox{  and  } \;\;\;\
\widetilde{E}(u)=E(u) \cap \mathcal{D}_q.
\]

We begin with a version of Lemma \ref{grazieli} for extended balls
restricted to $\mathcal{D}_q$. Consider the following subgroup of
$L_q^{3}$
$$
K=\{(\sigma ,\sigma ,\sigma ) \ : \ \sigma \in L_q \}=\{
(u_1,u_2,u_3) \mapsto (\lambda u_1, \lambda u_2, \lambda u_3) \ : \
\lambda \in \mathbb{F}_{q}^{*} \}.
$$
The commutative property $(\varphi,\sigma) \cdot
(\psi,\tau)=(\psi,\tau)\cdot(\varphi,\sigma)$ in $S_3 \ltimes K$
 yields that
$S_3 \ltimes K$ can be regard as the direct product $S_3 \times K$.

\begin{lemma}\label{grazieli2} The standard action of $S_3 \times K$ on $\mathbb{F}_{q}^{3}$ preserves
certain cardinalities of extended balls restricted to the set
$\mathcal{D}_q$:
\begin{enumerate}
\item $ \widetilde{E}(u^{\varphi}) = (\widetilde{E}(u))^{\varphi} $ for any $u$ in $\mathbb{F}_{q}^{3}$ and any $ \varphi  \in S_3 $.
\item $ \widetilde{E}(u^{\sigma}) = \widetilde{E}(u) $ for any $u$ in $\mathbb{F}_{q}^{3}$ and $ \sigma \in K $.
\item If $u$ and $v$ are in the same orbit of $\mathbb{F}_q^3$ by the
action $S_3 \times K$, then  $ \vert \widetilde{E}(u) \vert = \vert
\widetilde{E}(v) \vert $.
\end{enumerate}
\end{lemma}

\begin{proof}
The proofs are straightforward.
\end{proof}

What is the cardinality of $\widetilde{E}(u)$? Of course,
$\widetilde{E}(0)=\emptyset$.
 In  order to compute $\vert \widetilde{E}(u) \vert,$ we can assume that the first non-zero coordinate of
 $u$ is $1$, by Lemma
\ref{grazieli2}. The cardinality of  $\widetilde{E}(u) $ depends on
the weight of $u=(u_1,u_2,u_3)$ and the parameter $\delta(u)=\vert
\{u_1,u_2,u_3 \} \vert$, as described in the next statement.

\begin{theorem} \label{bolas}
Let $u$ be a vector in $\mathbb{F}_q^3$.
\begin{enumerate}
\item If $\omega(u)=1$ and $\delta(u)=2$, then $\vert \widetilde{E}(u)\vert =0$.
\item If $\omega(u)=2$ and $\delta(u)=2$, then $\vert \widetilde{E}(u)\vert =0$.
\item If $\omega(u)=2$ and $\delta(u)=3$, then $\vert \widetilde{E}(u)\vert = (q-1)(q-3)$.
\item If $\omega(u)=3$ and $\delta(u)=1$, then $\vert \widetilde{E}(u) \vert=0.$
\item If $\omega(u)=3$ and $\delta(u)=2$, then $\vert \widetilde{E}(u) \vert=(q-1)(2q-6)$.
\item If $\omega(u)=3$ and $\delta(u)=3$, then $\vert \widetilde{E}(u) \vert=(q-1)(3q-11)$.
\end{enumerate}
\end{theorem}

\begin{proof} {\bf Part 1}: if $\omega(u)=1$ and $\delta(u)=2$. Every scalar multiple of
$u$ contains at least two $0$. These multiples are not able to cover
any vector in  $\mathcal{D}_q$.

{\bf Part 2}: if $\omega(u)=2$ and $\delta(u)=2$. By Lemma
\ref{grazieli2}, we can assume $u=(0,1,1)$ without loss of
generality. Suppose for a contradiction that $v=(v_1,v_2,v_3) \in
\widetilde{E}(u)$, that is, there is $\lambda \in \mathbb{F}_q^{*}$
such that $d((0,\lambda,\lambda),(v_1,v_2,v_3)) \leq 1$. The
condition $(v_1,v_2,v_3) \in \mathcal{D}_q$ implies that $v_1 \neq
0$ and consequently $d((v_1,v_2,v_3),(0,\lambda,\lambda)) = 1$.
Since the vectors $(v_1,v_2,v_3)$ and $(0,\lambda,\lambda)$ differ
in the first coordinate, the absurd $v_2=v_3=\lambda$ is obtained.

{\bf Part 3}: if $\omega(u)=2$ and $\delta(u)=3$. Again by Lemma
\ref{grazieli2}, we can assume $u=(0,1,z)$, with $z \in
\mathbb{F}_q^{*}$ and $z \neq 1$. Let us prove first the following
statement.

\emph{Claim 1}: $\widetilde{B}(\lambda u) \cap
\widetilde{B}(\lambda' u) = \emptyset$ for distinct $\lambda,
\lambda' \in \mathbb{F}_q^{*}$.

Indeed, if $v=(v_1,v_2,v_3) \in \widetilde{B}(\lambda u)\cap
\widetilde{B}(\lambda' u)$, then $\lambda u$ and $v$ differ in the
first coordinate, and $\lambda' u$ and $v$ differ in the first
coordinate too, because $v_1 \neq 0$. Hence $v$ assumes both forms
$v=(v_1,\lambda,\lambda z)$ and $v=(v_1,\lambda',\lambda' z)$, and
$\lambda=\lambda'$. This leads an absurd.

If $v=(v_1,v_2,v_3) \in \widetilde{B}(\lambda u)$, then $v$ assumes
the form $v=(v_1,\lambda,\lambda z)$, where $v_1 \in
\mathbb{F}_q^{*}$, $v_1 \neq \lambda$ and $v_1 \neq \lambda z $.
Each one of these $(q-1)$ scalar multiples covers exactly $(q-3)$
vectors of $\mathcal{D}_q$. Claim 1 states that these $q-1$ sets are
pairwise disjoint, hence their union yields $\vert \widetilde{E}(u)
\vert =(q-1)(q-3)$.

Before proving the remaining parts, we need the following statement.

\emph{Claim 2}: Let $\omega(u)=3$. For distinct $\lambda, \lambda'
\in \mathbb{F}_q^{*}$, we claim that ${B}(\lambda u) \cap
{B}(\lambda' u)=\emptyset$. In particular,
\begin{equation}\label{isa}
\vert \widetilde{E}(u) \vert= (q-1) \vert \widetilde{B}(u) \vert.
\end{equation}

Suppose for a contradiction that ${B}(\lambda u) \cap {B}(\lambda'
u)\neq \emptyset$, thus the vectors $\lambda u$ and $\lambda' u$
agree in at least one coordinate, say $\lambda u_i=\lambda' u_i$.
Since $u_i \neq 0$, the condition $\lambda=\lambda'$ holds, which
leads an absurd.

{\bf Part 4}: if $\omega(u)=3$ and $\delta(u)=1$. Here $E(u) \cap
\mathcal{D}_q = \emptyset$, because each scalar multiple of $u$ has
three coincident coordinates too. There is not a vector of
$\mathcal{D}_q$ which is covered by some multiple of $u$.

{\bf Part 5}: if $\omega(u)=3$ and $\delta(u)=2$. By Lemma
\ref{grazieli2}, we can assume $u=(1,1,z)$ for some $z \in
\mathbb{F}_q^{*}$ and $z \neq 1$. Suppose that $v=(v_1,v_2,v_3) \in
\widetilde{B}(u)$. Since $v$ has three distinct coordinates and $u$
has two coincident coordinates, we obtain  $d(u,v) = 1$. The vector
$v$ assumes one of the
 forms: $(v_1,1,z)$ or $(1,v_2,z)$. Since $v_1,v_2 \in \mathbb{F}_q^{*}$ and $\{ v_1,v_2\} \cap \{1, z\} = \emptyset$, we can
choose $v_1$ and $v_2$ from $q-3$ distinct ways, thus $\vert
\widetilde{B}(u) \vert=2q-6$. Equation (\ref{isa}) implies $\vert
\widetilde{E}(u) \vert = (q-1)(2q-6)$.

{\bf Part 6}: if $\omega(u)=3$ and $\delta(u)=3$. We can choose
$u=(1,y,z)$, where $1$, $y \neq 0$, and $z \neq 0$ are pairwise
distinct. Let $v=(v_1,v_2,v_3)$ be a vector in $\widetilde{B}(u)$.
If $d(u,v)=0$ then $v=u$. If $d(u,v)=1$, then $v$ assumes one of the
forms: $(v_1,y,z)$, $(1,v_2,z)$, or $(1,y,v_3)$. Since
$\{v_1,v_2,v_3\} \cap \{0,1,y,z\}=\emptyset$, each variable $v_1$,
$v_2$, and $v_3$ can be chosen from $q-4$ possibilities. The
additive principle yields $\vert \widetilde{B}(u) \vert=3q-11$ and
Equation (\ref{isa}) concludes the counting.
\end{proof}

\section{Intersection of extended balls in $\mathcal{D}_q$} \label{Intersection}

Let us now focus on the behavior of $\widetilde{E}(u) \cap
\widetilde{E}(v)$, where $u,v$ are arbitrary vectors in
$\mathbb{F}_q^3$. Obviously $\vert \widetilde{E}(u)\cap
\widetilde{E}(v) \vert = 0$ whenever $\vert \widetilde{E}(u) \vert =
0$ or $\vert \widetilde{E}(v) \vert = 0$. On the other hand, if
$\vert \widetilde{E}(u) \cap \widetilde{E}(v) \vert \neq 0$, Theorem
\ref{bolas} implies that $u$ and $v$ must be of two types:
\begin{itemize}
  \item {\bf Type I}: vector of weight two with  three distinct coordinates.

  \item {\bf Type II}: vector of weight three.
\end{itemize}

It is well-known that  the additive group $\mathbb{Z}_{q-1}$ and the
multiplicative group $\mathbb{F}_q^{*}$ are isomorphic by the
relation $\overline{a} \mapsto \xi^a$, where $\xi$ denotes an
arbitrary generator of $\mathbb{F}_{q} ^{*}$. As usual, the class
$\overline{a} \in \mathbb{Z}_{q-1}$ is simply denoted by $a$, where
$0 \leq a \leq q-1$ and the multiplication follows the rule $ \xi^a
\xi^b=\xi^{a+b}=\xi^c, $ where $c=a+b$ in $\mathbb{Z}_{q-1}$.

\begin{example}\label{cachorro}
\emph{We illustrate here that the cardinality $\widetilde{E}(u) \cap
\widetilde{E}(v) $ can vary widely, even though $u$ and $v$ are
vectors of the type I or II.}

\emph{\begin{enumerate}
\item  Consider the vectors $u=(1,4,2)$ and $v=(1,3,4)$ in
$\mathbb{Z}_5^3$. Is is easy to see that $\vert \widetilde{E}(u)
\cap \widetilde{E}(v) \vert = 12$, since its intersection is formed
by
\[
\begin{array}{llll}
(1,3,2),(1,4,2),(3,4,2),(1,3,4),(2,3,4),(2,1,4),\\
(3,4,1),(3,2,1),(4,2,1),(2,1,3),(4,1,3),(4,2,3)
\end{array}.
\]
\item The intersection  $\widetilde{E}(u) \cap
\widetilde{E}(v)$ for $u=(1,2,4)$ and $v=(1,4,2)$ in
$\mathbb{Z}_5^3$ corresponds to
\[ \{ (1,4,3),(3,4,2),(1,3,4),(2,3,1),
(3,2,4),(4,2,1),(2,1,3),(4,1,2) \}.
\]
Thus $\vert \widetilde{E}(u) \cap \widetilde{E}(v) \vert = 8$.
\item
Take the vectors $u=(1,0,2)$ and $v=(1,2,3)$ in $\mathbb{Z}_5^3$. We
have
\[
\widetilde{E}(u) \cap \widetilde{E}(v) = \{
(1,3,2),(2,1,4),(3,4,1),(4,2,3) \}.
\]
\item Let $\xi$ be a primite element in $\mathbb{F}_4^{*}$.
Consider the vectors $u=(1,\xi^{1},\xi^{2})$ and
$v=(1,\xi^{2},\xi^{1})$ in $\mathbb{F}_4^3$. Checking case by case
the intersections $\widetilde{B}(\lambda u) \cap \widetilde{B}(\mu
v)$, where $\lambda,\mu \in \mathbb{F}_4^{*}$, is easy to see that
$\vert \widetilde{E}(u) \cap \widetilde{E}(v) \vert = 0$.
\item
Another pair of vectors which satisfies $\vert \widetilde{E}(u) \cap
\widetilde{E}(v) \vert = 0$ is $u=(1,2,4)$ and $v=(1,4,2)$ in
$\mathbb{Z}_7^3$.
\end{enumerate}}
\end{example}

\subsection{Intersection of extended balls reduced to intersection of balls} \label{aux}

\begin{remark}[Intersection of balls: a characterization]\label{uva}
\emph{Let $u=(u_1,u_2,u_3)$ and $v=(v_1,v_2,v_3)$ be arbitrary
vectors in $\mathbb{F}_q^3$. The set $B(u) \cap B(v)$ varies
according to the distance of the vectors, as follows: }

\emph{Case 1: if $d(u,v)=0$. Clearly $u=v$ and $B(u) \cap
B(v)=B(u)$.}

\emph{Case 2: if $d(u,v)=1$. There are three situations.
\begin{itemize}
  \item If $u_1 \neq v_1$, then $B(u) \cap B(v) = \{ (x,u_2,u_3) \ : \ x \in \mathbb{F}_q \}$.
  \item If $u_2 \neq v_2$, then $B(u) \cap B(v) = \{ (u_1,x,u_3) \ : \ x \in \mathbb{F}_q \}$.
  \item If $u_3 \neq v_3$,then $B(u) \cap B(v) = \{ (u_1,u_2,x) \ : \ x \in \mathbb{F}_q \}$.
\end{itemize}}

\emph{Case 3: if $d(u,v)=2$. Three subcases can hold.
\begin{itemize}
  \item If $u_1 = v_1$, then $B(u) \cap B(v) = \{ (u_1,v_2,u_3),(u_1,u_2,v_3) \}$.
  \item If $u_2 = v_2$, then $B(u) \cap B(v) = \{ (v_1,u_2,u_3),(u_1,u_2,v_3)  \}$.
  \item If $u_3 = v_3$, then $B(u) \cap B(v) = \{ (v_1,u_2,u_3),(u_1,v_2,u_3) \}$.
\end{itemize}}

\emph{Case 4: if $d(u,v)=3$, then clearly $B(u) \cap B(v) =
\emptyset$.}
\end{remark}

It is a little surprising that the computation of
 $\vert \widetilde{E}(u) \cap \widetilde{E}(v) \vert $ under the condition $u,v \in \mathcal{D}_q$ can be reduced to the
  cardinality of suitable intersections of balls. For this purpose, denote $\lambda
Z=\{\lambda z \, : \, z \in Z\}.$

\begin{lemma}\label{banana000}
Let $u$, $v$ be two vectors in $\mathbb{F}_q^3$, and $ \mu \in
\mathbb{F}_q$.
\begin{enumerate}
\item For every $\lambda \in \mathbb{F}_q^{*}$,
\[
 \widetilde{B}(\lambda u) \cap
\widetilde{B}(\mu v)  = \lambda [\widetilde{B}(u) \cap
\widetilde{B}(\lambda^{-1}\mu v)].
\]

\item  If the family $\{ \widetilde{B}(\lambda u) : \, \lambda \in \mathbb{F}_q^{*}\}$ is a partition of $\widetilde{E}(u)$ and
the family $\{ \widetilde{B}(\lambda v) : \, \lambda \in
\mathbb{F}_q^{*}\}$ is a partition of $\widetilde{E}(v)$, then
\[
\vert \widetilde{E}(u) \cap \widetilde{E}(v) \vert =
(q-1)\sum_{\mu\in\mathbb{F}_q^{*}}\vert
\widetilde{B}(u)\cap\widetilde{B}(\mu v)\vert.
\]
\end{enumerate}
\end{lemma}
\begin{proof}
\textbf{Part 1}: Note that $w \in \widetilde{B}(\lambda u)\cap
\widetilde{B}(\mu v)$ if and only if there are scalars $\alpha ,
\beta \in \mathbb{F}_q$ and canonical vectors $e_i, \; e_j $ such
that $w=\lambda u + \alpha e_i$ and $w=\mu v + \beta e_j$. These
equalities are equivalent to $\lambda^{-1}w = u + \lambda^{-1}\alpha
e_i$ and $\lambda^{-1}w= \lambda^{-1}\mu v + \lambda^{-1}\beta e_j$,
that is, $\lambda^{-1} w \in \widetilde{B}(u) \cap
\widetilde{B}(\lambda^{-1} \mu v)$.

\textbf{Part 2}: Part 1 and the fact that $\lambda (\cup Z_i)=\cup
(\lambda Z_i)$ produce
\begin{eqnarray}
\nonumber \widetilde{E}(u) \cap \widetilde{E}(v) & = &
\bigcup_{\lambda \in \mathbb{F}_q^{*}}
 \bigcup_{\mu \in \mathbb{F}_q^{*}}\widetilde{B}(\lambda u)\cap \widetilde{B}(\mu
 v) \nonumber\\
 & = &  \bigcup_{\lambda \in \mathbb{F}_q^{*}} \bigcup_{\mu \in \mathbb{F}_q^{*}} \lambda[\widetilde{B}(u)\cap \widetilde{B}(\lambda^{-1} \mu v)] = \nonumber \\
                                       & = &  \bigcup_{\lambda \in \mathbb{F}_q^{*}} \lambda
                                       \left[ \bigcup_{\mu \in \mathbb{F}_q^{*}}\widetilde{B}(u)\cap \widetilde{B}(\lambda^{-1} \mu v)\right]
                                        =\nonumber\\
                                        & = & \bigcup_{\lambda \in \mathbb{F}_q^{*}} \lambda
                                        \left[ \bigcup_{\mu \in \mathbb{F}_q^{*}}\widetilde{B}(u)\cap \widetilde{B}(\mu
                                        v)\right]. \label{abacate}
\end{eqnarray}
Since the sets in $\{ \widetilde{B}(\lambda u) : \, \lambda \in
\mathbb{F}_q^{*}\}$ are pairwise disjoint, we claim that
\begin{equation}\label{abacaxi}
\left\vert \bigcup_{\lambda \in \mathbb{F}_q^{*}}
 \lambda \left[ \bigcup_{\mu \in \mathbb{F}_q^{*}}\widetilde{B}(u)\cap \widetilde{B}(\mu v)\right] \right\vert =
  \sum_{\lambda \in \mathbb{F}_q^{*}} \left\vert \lambda\left[ \bigcup_{\mu \in \mathbb{F}_q^{*}}\widetilde{B}(u)\cap \widetilde{B}(\mu
  v)\right]\right\vert.
\end{equation}
Indeed, if there is $x \in \lambda \left[ \bigcup_{\mu \in
\mathbb{F}_q^{*}}\widetilde{B}(u)\cap \widetilde{B}(\mu v)\right]
\cap \lambda' \left[ \bigcup_{\mu \in
\mathbb{F}_q^{*}}\widetilde{B}(u)\cap \widetilde{B}(\mu v)\right]$
for $\lambda \neq \lambda'$, then in particular, $x \in
\widetilde{B}(\lambda u) \cap \widetilde{B}( \lambda' u)
=\emptyset$, an absurd.

For $\lambda \in \mathbb{F}_q^{*}$,
\begin{equation}\label{amora}
 \left\vert \lambda\left[
\bigcup_{\mu \in \mathbb{F}_q^{*}}\widetilde{B}(u)\cap
\widetilde{B}(\mu
  v)\right]\right\vert =
\left\vert \bigcup_{\mu \in\mathbb{F}_q^{*}}\widetilde{B}(u)\cap
\widetilde{B}(\mu v)\right\vert=\sum_{\mu \in \mathbb{F}_q^{*}}
\vert \widetilde{B}(u) \cap \widetilde{B}(\mu v) \vert,
\end{equation}
since $\widetilde{B}(u)\cap\widetilde{B}(\mu v) \subset
\widetilde{B}(\mu v)$ for all $\mu \in \mathbb{F}_q^{*}$, and the
sets in $\{ \widetilde{B}(\lambda v) : \; \lambda \in
\mathbb{F}_q^{*}\}$ are pairwise disjoint. From (\ref{abacate}),
(\ref{abacaxi}) and (\ref{amora}), we conclude
\[
\vert \widetilde{E}(u) \cap \widetilde{E}(v) \vert = \sum_{\lambda
\in \mathbb{F}_q^{*}} \left( \sum_{\mu \in \mathbb{F}_q^{*}} \vert
\widetilde{B}(u) \cap \widetilde{B}(\mu v) \vert \right) = (q-1)
\sum_{\mu \in \mathbb{F}_q^{*}} \vert \widetilde{B}(u) \cap
\widetilde{B}(\mu v) \vert.
\]
\end{proof}

\subsection{An auxiliary parameter}

\begin{definition}
\emph{Given arbitrary vectors $u$ and $v \in \mathbb{F}_q^3$, define
\[
\rho_q(u,v) =\left\{
\begin{array}{ll}
0 & \mbox{ if } \vert \widetilde{E}(u) \vert = 0$ { \mbox or } $\vert \widetilde{E}(v)\vert = 0\\
\sum_{\mu\in\mathbb{F}_q^{*}}\vert
\widetilde{B}(u)\cap\widetilde{B}(\mu v)\vert & \mbox{ otherwise.}
\end{array}
\right.
\]}
\end{definition}

\begin{corollary} \label{banana001}
For arbitrary vectors $u,$ $v$ in $\mathbb{F}_q^3$,
\[
\vert \widetilde{E}(u) \cap \widetilde{E}(v) \vert =
\rho_q(u,v)(q-1).
\]
\end{corollary}

\begin{proof}
The statement is obvious when $\vert \widetilde{E}(u) \vert = 0$ or
$\vert \widetilde{E}(v) \vert = 0$. Otherwise,  $\vert
\widetilde{E}(u) \vert \neq 0$, $\vert \widetilde{E}(v) \vert \neq
0$, and both vectors are of type I or II. Hence
$\widetilde{B}(\lambda u) \cap \widetilde{B}(\lambda' u) =
\emptyset$ and $\widetilde{B}(\lambda v) \cap \widetilde{B}(\lambda'
v) = \emptyset$ for all $\lambda \neq \lambda'$ in
$\mathbb{F}_q^{*}$, according to Claims 1 and 2 of the proof in
Theorem \ref{bolas}. Lemma \ref{banana000} concludes the statement.
\end{proof}

\begin{example}\label{ozzy}
\emph{Example \ref{cachorro} and Corollary \ref{banana001}
illustrate a few sharp values:
$$
\begin{array}{llll}
\rho_5((1,4,2),(1,3,4))=3, & \rho_5((1,2,4),(1,4,2))=2, & \rho_5((1,0,2),(1,2,3))=1, \\
\rho_4((1,\xi^{1},\xi^{2}),(1,\xi^{2},\xi^{1}))=0, &
\rho_7((1,2,4),(1,4,2))=0.
\end{array}
$$}
\end{example}

\subsection{The computation of $\rho_q(u,v)$}

In this subsection we are concerned with the computation of
$\rho_q(u,v)$ for arbitrary vectors $u,v \in \mathcal{D}_q$.

\begin{example}
\emph{Let $u=(2,0,5)$ and $v=(6,7,9)$ be vectors in
$\mathbb{Z}_{11}^3$. Since $u$ is a vector of the type I and $v$ is
a vector of type II, Corollary \ref{banana001} yields
$\rho_{11}(u,v) = \sum_{\mu \in \mathbb{Z}_{11}^{*}} \vert
\widetilde{B}(u) \cap \widetilde{B}(\mu v) \vert$. Clearly $\vert
\widetilde{B}(u) \cap \widetilde{B}(\mu v) \vert = 0$ when $d(u,\mu
v) = 3$. The scalars $\mu$ which satisfy $d(u,\mu v) \leq 2$ are:
$0$, $3$, and $4$. Thus
\[
\rho_{11}(u,v) = \vert \widetilde{B}(u) \cap \widetilde{B}(0) \vert
+ \vert \widetilde{B}(u) \cap \widetilde{B}(3 v) \vert + \vert
\widetilde{B}(u) \cap \widetilde{B}(4 v) \vert.
\]
A simple inspection reveals that $ \widetilde{B}(u) \cap
\widetilde{B}(0) = \emptyset$, $\widetilde{B}(u) \cap
\widetilde{B}(3 v)=\{ (2,10,5)\}$, and $\widetilde{B}(u) \cap
\widetilde{B}(4 v) = \{ (2,6,5)\}.$ Therefore
$\rho_{11}((2,0,5),(6,7,9))=2$.}
\end{example}

The example above illustrates a curious but important fact: in order
to compute $\rho_q(u,v),$ we do not have to verify $\vert
\widetilde{B}(u) \cap \widetilde{B}(\mu v) \vert$ for all $\mu \in
\mathbb{F}_q^{*}$. Indeed, it is sufficient to evaluate $\vert
\widetilde{B}(u) \cap \widetilde{B}(\mu v) \vert$ for at most three
instances of $\mu \in \mathbb{F}_q^{*}$, according to the next
result.

\begin{lemma}\label{estranged}
Given $u=(\xi^{a},\xi^{b},\xi^{c})$ and
$v=(\xi^{d},\xi^{e},\xi^{f})$ in $\mathcal{D}_q$,
\[
\rho_q(u,v) = \vert \widetilde{B}(u) \cap \widetilde{B}(\xi^{a-d}v)
\vert + \vert \widetilde{B}(u) \cap \widetilde{B}(\xi^{b-e}v) \vert
+ \vert \widetilde{B}(u) \cap \widetilde{B}(\xi^{c-f}v) \vert.
\]
\end{lemma}

\begin{proof}
Since $u$ and $v$ are vectors of type II, Corollary \ref{banana001}
implies
\[
\rho_q(u,v) = \sum_{\mu\in\mathbb{F}_q^{*}}\vert
\widetilde{B}(u)\cap\widetilde{B}(\mu v)\vert.
\]
We analyze now the contribution of each scalar ${\mu}$. For a scalar
$\mu$ such that $d(u,\mu v)=3$, Remark \ref{uva} implies $\vert B(u)
\cap B(\mu v) \vert = 0$, and consequently  $\vert \widetilde{B}(u)
\cap \widetilde{B}(\mu v) \vert = 0$. It remains the case where
$d(u,\mu v) \leq 2$, which produces the following possibilities for
$\mu $: $\xi^{a-d}$, $\xi^{b-e}$, and $\xi^{c-f}$.
\end{proof}

\begin{theorem}\label{icarus}
Let $q$ be a prime power and $u,v \in \mathcal{D}_q$. The following
characterization holds:
\begin{enumerate}
  \item $\rho_q(u,v)=0$ if and only if $u=\lambda (1,\xi^{a},\xi^{b})$, $v=\mu (1,\xi^{b},\xi^{a})$ for some $\lambda,\mu \in \mathbb{F}_q^{*}$,
  where $a$, $b$ are distinct, non-zero, $2a=b$, and $2b=a$.
  \item $\rho_q(u,v) \geq 2$, otherwise.
\end{enumerate}
\end{theorem}

\begin{proof}
{\bf Part 1}: We can assume without lost of generality that $u$ and
$v$ have the first coordinate equal to $1$, that is,
$u=(1,\xi^{a},\xi^{b})$, $v=(1,\xi^{b},\xi^{a})$, where $a$, $b$ are
distinct, non-zero, $2a=b$ and $2b=a$. By Lemma \ref{estranged},
\[
\rho_q(u,v) = \vert \widetilde{B}(u) \cap \widetilde{B}(v) \vert +
\vert \widetilde{B}(u) \cap \widetilde{B}(\xi^{a-b}v) \vert + \vert
\widetilde{B}(u) \cap \widetilde{B}(\xi^{b-a}v) \vert.
\]
The characterization in Remark \ref{uva} shows us that
\begin{eqnarray*}
B(u) \cap B(v)           & = & \{ (1,\xi^{b},\xi^{b}),(1,\xi^{a},\xi^{a}) \} ,\\
B(u) \cap B(\xi^{a-b} v)  & = & \{(\xi^{a-b},\xi^{a},\xi^{b}),(1,\xi^{a},\xi^{2a-b}) \},\\
B(u) \cap B(\xi^{b-a} v)  & = &
\{(\xi^{b-a},\xi^{a},\xi^{b}),(1,\xi^{2b-a},\xi^{b}) \}.
\end{eqnarray*}
Since $2a=b$ and $2b=a$, we obtain $\widetilde{B}(u) \cap
\widetilde{B}(v) = \emptyset$, $\widetilde{B}(u) \cap
\widetilde{B}(\xi^{a-b}v) = \emptyset$, and $\widetilde{B}(u) \cap
\widetilde{B}(\xi^{b-a}v)= \emptyset$. Therefore $\rho_q (u,v) = 0$.

{\bf Part 2}: It is enough to prove that $\rho_q(u,v) \geq 2$ for
the following situations:
\begin{itemize}
  \item[(i)] $u=(1,\xi^a,\xi^b)$, $v=(1,\xi^c,\xi^d)$,
  \item[(ii)]
$u=(1,\xi^a,\xi^b)$, $v=(1,\xi^a,\xi^c)$,
  \item[(iii)] $u=(1,\xi^a,\xi^b)$,
$v=(1,\xi^c,\xi^b)$,
  \item[(iv)] $u=(1,\xi^a,\xi^b)$, $v=(1,\xi^c,\xi^a)$,
  \item[(v)] $u=(1,\xi^a,\xi^b)$, $v=(1,\xi^b,\xi^a)$, with $2a \neq
  b$ or $2b \neq a$,
\end{itemize}
where the elements $a,b,c,d \in \mathbb{Z}_{q-1}$ are pairwise
distinct and non-zero.

{\it Item (i)}: Here $ \rho_q(u,v)  = \vert \widetilde{B}(u) \cap
\widetilde{B}(v) \vert + \vert \widetilde{B}(u) \cap
\widetilde{B}(\xi^{a-c}v) \vert + \vert \widetilde{B}(u) \cap
\widetilde{B}(\xi^{b-d}v) \vert$. Since $d(u,v)=2$, it follows by
Remark \ref{uva} that
\[
B(u) \cap B(v)          = \{ (1,\xi^{c},\xi^{b}),
(1,\xi^{a},\xi^{d}) \},
\]
which is a subset of $\mathcal{D}_q$, thus $\rho_q (u,v)\geq 2$.

{\it Item (ii)}: Note that $\rho_q(u,v)  = \vert \widetilde{B}(u)
\cap \widetilde{B}(v) \vert + \vert \widetilde{B}(u) \cap
\widetilde{B}(\xi^{b-c}v) \vert$. Apply Remark \ref{uva} when
$d(u,v)=1$. Since
\[
B(u) \cap B(v)         = \{ (1,\xi^{a},x) \ : \ x \in \mathbb{F}_q \} ,\\
\]
is a subset of $\mathcal{D}_q$, $\vert \widetilde{B}(u) \cap
\widetilde{B}(v) \vert =q-3$ holds, and $\rho_q(u,v) \geq q-3$
follows as a consequence.

{\it Item (iii)} This case can be proved as an immediate consequence
of item (ii) and the concept of $\mathbb{F}_q$-equivalence.

{\it Item (iv)} Lemma \ref{estranged} implies
$$\rho_q(u,v)  = \vert
\widetilde{B}(u) \cap \widetilde{B}(v) \vert + \vert
\widetilde{B}(u) \cap \widetilde{B}(\xi^{a-c}v) \vert + \vert
\widetilde{B}(u) \cap \widetilde{B}(\xi^{b-a}v) \vert. $$
 We consider two cases.

Case 1: if $a-c \neq b-a$. Since $d(u,v)=2$, $d(u,\xi^{a-c}v)=2$,
and $d(u,\xi^{b-a}v)=2$, Remark \ref{uva} gives us
\begin{eqnarray*}
B(u) \cap B(v)           & = & \{ (1,\xi^{c},\xi^{b}),(1,\xi^{a},\xi^{a}) \} ,\\
B(u) \cap B(\xi^{a-c} v)  & = & \{(\xi^{a-c},\xi^{a},\xi^{b}),(1,\xi^{a},\xi^{2a-c}) \},\\
B(u) \cap B(\xi^{b-a} v)  & = &
\{(\xi^{b-a},\xi^{a},\xi^{b}),(1,\xi^{b+c-a},\xi^{b}) \}.
\end{eqnarray*}
We still need to analyze more two subcases. It is easy to check that
$\vert \widetilde{B}(u) \cap \widetilde{B}(v) \vert =1$ for all
subcases below.

Subcase 1.1: If $2a \neq c$, then $(1,\xi^{a},\xi^{2a-c}) \in
\mathcal{D}_q$ and $\vert \widetilde{B}(u) \cap
\widetilde{B}(\xi^{a-c} v) \vert \geq 1$.

Subcase 1.2: If $2a = c$, then $(\xi^{b-a},\xi^{a},\xi^{b}) \in
\mathcal{D}_q$ and $\vert \widetilde{B}(u) \cap
\widetilde{B}(\xi^{b-a} v) \vert \geq 1$.

Therefore $\rho_q(u,v) \geq 2$ in both subcases.

Case 2: if $a-c = b-a$. Here $d(u,\xi^{a-c}v)=1$. By Remark
\ref{uva},
\[
B(u) \cap B(\xi^{a-c} v)  = \{(x,\xi^{a},\xi^{b}) \ : \ x \in
\mathbb{F}_q \}.
\]
Thus $\vert \widetilde{B}(u) \cap \widetilde{B}(v) \vert =1$ and
$\vert \widetilde{B}(u) \cap \widetilde{B}(\xi^{a-c}v) \vert = q-3$,
which implies $\rho_q(u,v) \geq q-3$.

{\it Item (v)} In this case,
$$\rho_q(u,v)  = \vert \widetilde{B}(u)
\cap \widetilde{B}(v) \vert + \vert \widetilde{B}(u) \cap
\widetilde{B}(\xi^{a-b}v) \vert + \vert \widetilde{B}(u) \cap
\widetilde{B}(\xi^{b-a}v) \vert.
$$
 We divide the proof into two cases.

Case 1: if $a-b \neq b-a$. Since $d(u,\xi^{a-b}v)=2$ and
$d(u,\xi^{b-a}v)=2$, Remark \ref{uva} implies
\begin{eqnarray*}
B(u) \cap B(\xi^{a-b} v)  & = & \{(\xi^{a-b},\xi^{a},\xi^{b}),(1,\xi^{a},\xi^{2a-b}) \},\\
B(u) \cap B(\xi^{b-a} v)  & = &
\{(\xi^{b-a},\xi^{a},\xi^{b}),(1,\xi^{2b-a},\xi^{b}) \}.
\end{eqnarray*}

Subcase 1.1: if $2a \neq b$. The vectors $(1,\xi^{a},\xi^{2a-b})$
and $(\xi^{b-a},\xi^{a},\xi^{b})$ belong to $\mathcal{D}_q$. Hence
$\vert \widetilde{B}(u) \cap \widetilde{B}(\xi^{a-b}v) \vert \geq 1$
and  $\vert \widetilde{B}(u) \cap \widetilde{B}(\xi^{b-a}v) \vert
\geq 1$. It means that $\rho_q(u,v) \geq 2$.

Subcase 1.2: if $2b \neq a$. The vectors
$(\xi^{a-b},\xi^{a},\xi^{b})$ and $(1,\xi^{2b-a},\xi^{b})$ belong to
$\mathcal{D}_q$. Since $\vert \widetilde{B}(u) \cap
\widetilde{B}(\xi^{a-b}v) \vert \geq 1$ and $\vert \widetilde{B}(u)
\cap \widetilde{B}(\xi^{b-a}v) \vert \geq 1$, the bound $\rho_q(u,v)
\geq 2$ holds.

Case 2: if $a-b = b-a$. Because  $d(u,\xi^{a-b}v)=1$, Remark
\ref{uva} implies
\[
B(u) \cap B(\xi^{a-b} v)  = \{(x,\xi^{a},\xi^{b}) \ : \ x \in
\mathbb{F}_q \}.
\]
It is easy to check that $\vert \widetilde{B}(u) \cap
\widetilde{B}(\xi^{a-b}v) \vert = q-3$, and $\rho_q(u,v) \geq q-3$
follows.
\end{proof}

\section{Proof of Theorem \ref{dream}} \label{prova}

\begin{definition}\label{problemaextremal}
\emph{Given a prime power $q$, define
\[
\rho(q) = \min \{ \rho_q(u,v) \ : \ u, v \in \mathcal{D}_q \}.
\]}
\end{definition}

\begin{example}
\emph{We obtain immediately from Example \ref{ozzy} that
$\rho(4)=0$, $\rho(5) \leq 2$ and $\rho(7)=0$.}
\end{example}

The parameter $\rho(q)$ is completely determined, according to the
next statement.

\begin{theorem}\label{forever}
For a prime power $q$,
\[
\rho(q) =\left\{
\begin{array}{l}
0 \mbox{ if  } 3 \mbox{ divides } q-1,\\
2 \mbox{  otherwise.}
\end{array}
\right.
\]
\end{theorem}

\begin{proof}
{\bf Part 1}: If $3$ divides $q-1$, then there is a non-zero $k \in
\mathbb{Z}$ such that $3k=q-1$, that is, $3k=0$ in the ring
$\mathbb{Z}_{q-1}$. Since the vectors $u=(1,\xi^{k},\xi^{-k})$ and
$v=(1,\xi^{-k},\xi^{k})$ satisfy the hypothesis of Proposition
\ref{icarus}, the value $\rho_q (u,v)=0$ holds.

{\bf Part 2}: If $3$ does not divide $q-1$, then there are distinct
numbers  $a,b \in \mathbb{Z}_{q-1}^{*}$ such that $2a=b$ and $2b=a$.
An application of Proposition \ref{icarus} yields $\rho_q (u,v) \geq
2$ for all $u,v \in \mathcal{D}_q$, that is, $\rho(q) \geq 2$.

Choose an element $a \in \mathbb{Z}_{q-1}$ such that $a \neq 0$, $2a
\neq 0$ and $3a \neq 0$. We consider the vectors
$u=(1,\xi^{a},\xi^{2a})$ and $v=(1,\xi^{2a},\xi^{a})$. Thus
$\rho_q(u,v)  = \vert \widetilde{B}(u) \cap \widetilde{B}(v) \vert +
\vert \widetilde{B}(u) \cap \widetilde{B}(\xi^{-a}v) \vert + \vert
\widetilde{B}(u) \cap \widetilde{B}(\xi^{a}v) \vert $, and by Remark
\ref{uva},
\begin{eqnarray*}
B(u) \cap B(v)            & = & \{ (1,\xi^{2a},\xi^{2a}),(1,\xi^{a},\xi^{a}) \} ,\\
B(u) \cap B(\xi^{-a} v)   & = & \{ (\xi^{-a},\xi^{a},\xi^{2a}),(1,\xi^{a},1) \},\\
B(u) \cap B(\xi^{a} v)    & = &
\{(\xi^{a},\xi^{a},\xi^{2a}),(1,\xi^{3a},\xi^{2a}) \}.
\end{eqnarray*}
Clearly, the vectors $(1,\xi^{2a},\xi^{2a})$, $(1,\xi^{a},\xi^{a})$,
$(1,\xi^{a},1)$ and $(\xi^{a},\xi^{a},\xi^{2a})$ do not belong to
$\mathcal{D}_q$. Both vectors $(\xi^{-a},\xi^{a},\xi^{2a})$ and
$(1,\xi^{3a},\xi^{2a})$ belong to $\mathcal{D}_q$. Hence $\vert
\widetilde{B}(u) \cap \widetilde{B}(v) \vert = 0$, $\vert
\widetilde{B}(u) \cap \widetilde{B}(\xi^{-a} v) \vert = 1$, $\vert
\widetilde{B}(u) \cap \widetilde{B}(\xi^{a} v) \vert = 1$, and
consequently $\rho_q (u,v) = 2$.
\end{proof}

\begin{proof}[Proof of Theorem \ref{dream}]:
Since $\mathcal{E} = \{ \widetilde{E}(u) \ : \ u \in \mathcal{D}_q
\}$,  Corollary \ref{banana001} reveals that $(q-1)\rho(q)$ is the
maximum $t$ such that the family $\mathcal{E}$ is  $t$-intersecting.
Thus the computation of $\theta(\mathcal{E})$ is reduced to Theorem
\ref{forever}.
\end{proof}

\section{Lower bounds of short coverings}\label{lower bounds}

\subsection{Necessary conditions for a short covering}

Some necessary conditions for a short co\-ve\-ring with ``few
vectors" are established here. For sake this purpose, let
$\pi_j(u_1,u_2,u_3)=u_j$
 denote the $j$-th canonical projection of $\mathbb{F}_q^3$ into  $\mathbb{F}_q$, where $1 \leq j \leq 3$.
 The symbol $*$ represents an arbitrary element in $\mathbb{F}_q$.

\begin{theorem}\label{caracterizacao}
Given a prime power $q \geq 7$, let $m = \lceil (q+1)/2 \rceil$.
Suppose that $\mathcal{H}=\{ h_1,\ldots,h_m \}$ is a short covering
of $\mathbb{F}_q^3$. The following conditions hold:
\begin{enumerate}
  \item There must be at least a vector in $\mathcal{H}$ with weight 3.
  \item For each coordinate $j$, $1 \leq j \leq 3$, there
must be at least a  vector $h_k \in \mathcal{H}$  such that
$\pi_j(h_k)=0$.
 \item The set $\mathcal{H}$ is $\mathbb{F}_q$-equivalent to one of the sets:
\begin{eqnarray*}
\mathcal{H}_1 & = & \{ (1,1,1),(0,*,*),(*,0,*),(*,*,0),h_5,\ldots,h_m \},\\
\mathcal{H}_2 & = & \{ (1,1,1),(0,*,*),(*,0,0),h_4,\ldots,h_m \}.
\end{eqnarray*}
\end{enumerate}
\end{theorem}

\begin{proof}
\textbf{Part 1}: If each vector in $\mathcal{H}$ has weight at most
$2$, Theorem \ref{bolas} yields $\vert \widetilde{E}(h_i) \vert \leq
(q-1)(q-3)$ for every $i$, $1 \leq i \leq m$. Thus the set
$\mathcal{H}$ is able to cover at most $m(q-1)(q-3)$ vectors of
$\mathcal{D}_q$. Because $\vert \mathcal{D}_q
\vert=(q-1)(q-2)(q-3)$, the set $\mathcal{H}$ is not a short
covering of $\mathbb{F}_q^3$, when $q \geq 7$.

\textbf{Part 2}: Assume without loss of generality that
$\omega(h_1)=3$. We also suppose $h_1=(1,1,1)$, by
$\mathbb{F}_q$-equivalence. Consider the plane
\[
\Pi_1= \{ (0,u_2,u_3) \ : \ u_2,u_3 \in \mathbb{F}_q \}
\]
and its subset  $\mathcal{X}_1=\{ (0,u_2,u_3) \in \Pi_1 \ : \ u_2
\neq u_3 \mbox{ and } u_2,u_3 \neq 0 \}$. The heart of the proof
consists in checking that $\mathcal{H}$ is not able to cover
(shortly) all the plane $\Pi_1$. Since $E(h_1) \cap \mathcal{X}_1=
\emptyset$, the whole set $\mathcal{X}_1$ must be covered by
$\{h_2,\ldots,h_m\}$. Suppose for a contradiction that $\pi_1(h_2)
\neq 0, \ldots, \pi_1(h_m) \neq 0$. Each one of the vectors in
$\{h_2,\ldots,h_m\}$ covers at most $q-1$ vectors of
$\mathcal{X}_1$, thus
$$\vert [ E(h_2)\cup \cdots \cup E(h_m)] \cap
\mathcal{X}_1 \vert \leq (m-1)(q-1).
$$
From the fact that $m = (q+1)/2$ if $q$ is odd and $m = (q+2)/2$ if
$q$ is even,
\[
(m-1)(q-1) = \frac{q}{2}(q-1)<(q-1)(q-2)=|\mathcal{X}_1|
\]
holds for every $q \geq 5$. Hence $\mathcal{X}_1 \not\subset E(h_2)
\cup \cdots \cup E(h_m)$. The statement for the case $j=1$ is
proved. The argument for $j=2$ and $j=3$ follows analogously.

\textbf{Part 3}: It is a consequence of both Parts 1 and 2. There is
a vector $h_1$ in $\mathcal{H}$ with $\omega(h_1)=3$. We also assume
$h_1=(1,1,1)$, by $\mathbb{F}_q$-equivalence. The Part 2 implies
that for each coordinate $j$, $1 \leq j \leq 3$, there must be at
least a vector $h_k \in \mathcal{H}$ such that $\pi_j(h_k)=0$. Thus
there are three vectors of  type $(0,*,*)$, $(*,0,*)$, $(*,*,0)$ in
$\mathcal{H}$ or there are two vectors of the type $(0,*,*)$,
$(*,0,0)$ in $\mathcal{H}$. The first case yields that $\mathcal{H}$
and $\mathcal{H}_1$ are $\mathbb{F}_q$-equivalent, while the second
case implies that $\mathcal{H}$ and $\mathcal{H}_2$ are
$\mathbb{F}_q$-equivalent.
\end{proof}

\subsection{Sketch}

The rest of this section is concerned with the computation of lower
bounds on $c(q)$, where $7 \leq q \leq 9$. The condition $c(q) > m$
corresponds to the statement: neither of the $ \binom{q^{3}}{m}$
$m$-subsets of $\mathbb{F}_{q}^3$, $\mathcal{H}$ satisfies the
equation (\ref{e1}).

Since the search space is often huge and the extended balls are
highly intersecting, it is not so accurate checking Eq. (\ref{e1})
straightforwardly. A sharp approach essentially analyzes the
be\-ha\-vior of the extended balls in $\mathcal{D}_q$. A little more
precise, the idea is described briefly as follows.

Given $q$, suppose by absurd that there is a short covering
$\mathcal{H}=\{h_1,\ldots,h_m \}$ of $\mathbb{F}_q^3$ with $m=
\lceil (q+1)/2 \rceil$ vectors. Theorem \ref{caracterizacao} states
 that there are only two possibilities for $\mathcal{H}$. Since
$\mathcal{H}$ is also a short covering of the subset
$\mathcal{D}_q$, the condition $\mathcal{D}_q \subset \cup_{i=1}^{m}
\widetilde{E}(h_i)$ holds. On the other hand, if we show that
\begin{equation} \label{e2}
\left\vert \bigcup_{i=1}^{m} \widetilde{E}(h_i) \right\vert <
(q-1)(q-2)(q-3)
\end{equation}
then an absurd raises: $\mathcal{D}_q$ is not contained in
$\cup_{i=1}^{m} \widetilde{E}(h_i)$. For the cases $q=7$ and $q=9$,
the stronger condition
\begin{equation} \label{e3}
\sum_{i=1}^{m} \vert \widetilde{E}(h_i) \vert < (q-1)(q-2)(q-3)
\end{equation}
is sufficient to show that (\ref{e2}) is valid.

\subsection{New lower bounds}

\begin{proposition}\label{inferior7} We obtain $c(7) \geq 5$.
\end{proposition}

\begin{proof}
Suppose by a contradiction that $\mathcal{H}=\{ h_1,\ldots,h_4 \}$
is a short covering of $\mathbb{F}_7^3$. Theorem
\ref{caracterizacao} states that there are only two forms for
$\mathcal{H}$, namely:
\begin{eqnarray*}
\mathcal{H}_1 & = & \{ (1,1,1),(0,*,*),(*,0,*),(*,*,0) \},\\
\mathcal{H}_2 & = & \{ (1,1,1),(0,*,*),(*,0,0),(*,*,*) \}.
\end{eqnarray*}

If $\mathcal{H}=\mathcal{H}_1$, then Theorem \ref{bolas} yields
$\vert \widetilde{E}(h_1) \vert =0$ and $\vert \widetilde{E}(h_i)
\vert \leq 24$ for $i \in \{2,3,4\}$. Thus $\mathcal{H}_1$ covers at
most $72$ vectors in $\mathcal{D}_7$.

Otherwise, $\mathcal{H}=\mathcal{H}_2$. Theorem \ref{bolas} implies
that $\vert \widetilde{E}(h_1) \vert =0$, $\vert \widetilde{E}(h_2)
\vert \leq 24$, $\vert \widetilde{E}(h_3) \vert =0$ and $\vert
\widetilde{E}(h_4) \vert \leq 60$. Hence $\mathcal{H}_2$ covers at
most $84$ vectors in $\mathcal{D}_7$.

Since $\vert \mathcal{D}_7 \vert=120$, the inequality (\ref{e3})
holds. We conclude that neither $\mathcal{H}_1$ nor $\mathcal{H}_2$
could cover all the space $\mathcal{D}_7$. Thus $c(7) \geq 5$.
\end{proof}

\begin{proposition}\label{inferior9}
The lower bound $c(9) \geq 6$ holds.
\end{proposition}
\begin{proof}
Suppose for a contradiction that ${\mathcal{H}}=\{ h_1,\ldots,h_5
\}$ is a short covering of $\mathbb{F}_9^3$. From Theorem
\ref{caracterizacao}, the set $\mathcal{H}$ can be
$\mathbb{F}_q$-equivalent to only two forms:
\begin{eqnarray*}
\mathcal{H}_1  & = & \{ (1,1,1),(0,*,*),(*,0,*),(*,*,0),(*,*,*) \},\\
\mathcal{H}_2& = & \{ (1,1,1),(0,*,*),(*,0,0),(*,*,*),(*,*,*) \}.
\end{eqnarray*}

If $\mathcal{{H}}=\mathcal{{H}}_1$, then Theorem \ref{bolas} yields
$\vert \widetilde{E}(h_1) \vert=0$, $\vert \widetilde{E}(h_i) \vert
\leq 48$ for $i \in \{ 2,3,4 \}$, and $\vert \widetilde{E}(h_5)
\vert \leq 128$. The set $\mathcal{H}_1$ covers at most $272$
vectors of $\mathcal{D}_9$.

If $\mathcal{{H}}=\mathcal{{H}}_2$, Theorem \ref{bolas} produces
$\vert \widetilde{E}(h_1) \vert =0$, $\vert \widetilde{E}(h_2) \vert
\leq 48$, $\vert \widetilde{E}(h_3) \vert =0$ and $\vert
\widetilde{E}(h_i) \vert \leq 128$ for $i \in \{4,5\}$. Hence
$\mathcal{H}_2$ covers at most $304$ vectors of $\mathcal{D}_9$.

Since $\vert \mathcal{D}_9 \vert = 336$, the inequality (\ref{e3})
is satisfied here; neither $\mathcal{H}_1$ nor $\mathcal{{H}}_2$ is
a short covering of $\mathcal{{D}}_9$.
\end{proof}

The argument for $c(8) >5 $ is a little more intricate than the
previous ones, because the inequality (\ref{e3}) does not hold for
all candidates $\mathcal{H}$. The search space corresponds to the
family of all $5$-subsets of $\mathbb{F}_{8}^3$, with
$\binom{8^{3}}{5}\simeq 2.8 \times 10^{11}$ candidates. Theorem
\ref{bolas} is not enough powerful to deal with all candidates.
Therefore we shall apply Theorem \ref{dream} too.

\begin{proposition}\label{inferior8}
The bound $c(8) \geq 6$ holds.
\end{proposition}

\begin{proof}
Suppose for a contradiction that  $\mathcal{H}=\{ h_1,\ldots,h_5 \}$
is a short covering of $\mathbb{F}_8^3$. Now, by Theorem
\ref{caracterizacao}, we can assume that $\mathcal{{H}}$ has one of
two possible forms:
\begin{eqnarray*}
\mathcal{{H}}_1 & = & \{ (1,1,1),(0,*,*),(*,0,*),(*,*,0),(*,*,*) \},\\
\mathcal{{H}}_2 & = & \{ (1,1,1),(0,*,*),(*,0,0),(*,*,*),(*,*,*) \}.
\end{eqnarray*}

An application of Theorem \ref{bolas} to the case
$\mathcal{H}=\mathcal{H}_1$ yields $\vert \widetilde{E}(h_1) \vert
=0$, $\vert \widetilde{E}(h_i) \vert \leq 35$ for $i \in \{2,3,4\}$,
and $\vert \widetilde{E}(h_5) \vert \leq 91$. Hence $\mathcal{H}_1$
covers at most $196$ vectors in $\mathcal{D}_8$. Since $\vert
\mathcal{D}_8 \vert=210$, the inequality (\ref{e3}) holds, and thus
$\mathcal{H}_1$ could not be a short covering of $\mathcal{D}_8$.

On the other hand, computation of the bounds from Theorem
\ref{bolas} for the case $\mathcal{H}=\mathcal{H}_2$ shows us $\vert
\widetilde{E}(h_1) \vert =0$, $\vert \widetilde{E}(h_2) \vert \leq
35$, $\vert \widetilde{E}(h_3) \vert =0$, and $\vert
\widetilde{E}(h_i) \vert \leq 91$ for $i \in \{4,5\}$. Hence
$\mathcal{H}_2$ covers at most $217$ vectors in $\mathcal{D}_8$.

\emph{Claim}: The vectors $h_4$ and $h_5$ belong to $\mathcal{D}_8.$

Indeed, if $h_4 \not\in \mathcal{D}_8$, Theorem \ref{bolas} implies
$\vert \widetilde{E}(h_4) \vert \leq 70$. Hence $\mathcal{H}_2$
covers at most $196$ vectors of $\mathcal{D}_{8}$, and
$\mathcal{H}_2$ is not a short covering of $\mathcal{D}_{8}$, since
$\vert \mathcal{D}_8 \vert=210$. If $h_5$ does not satisfy these
conditions, we obtain an absurd analogously.

By Theorem \ref{dream}, $\vert \widetilde{E}(h_4) \cap
\widetilde{E}(h_5) \vert \geq 14$, and thus $\mathcal{H}_2$ is a
short covering of at most $203$ vectors of $\mathcal{D}_8$, thus the
inequality (\ref{e2}) holds.
\end{proof}

\section{Construction of short coverings} \label{bounds}

\subsection{From actions of groups to coverings}

A systematical way of finding good short coverings is described in
\cite[Theorem 1]{irene} on the basis of invariant sets under
suitable actions. An adaptation of this method is described below.

The standard action of $G=S_3 \times K$ on $\mathbb{F}_{q}^{3}$
plays a central role in our results. The set
$$
\mathcal{A}_{q}=\{(u_1,u_2,u_3) \in \mathbb{F}_q^3 \ : \ u_1,u_2,u_3
\mbox{ are pairwise distinct} \}.
$$
is invariant by the action of the direct group $S_{3} \times K$,
which has two orbits, namely, $\{u \in \mathcal{A}_q \ : \
d(u,0)=3\}$ and $\{u\in \mathcal{A}_q \ : \ d(u,0)=2\}$.

\begin{theorem}\label{metodo}
Let $N$ be a subgroup of $S_{3}$ and choose a subset $\mathcal{L}$
of $\mathbb{F}_{q}^{3}$ which is invariant under the action of $N$,
that is, $\mathcal{L}^{N}=\mathcal{L}$. Let $\mathcal{O}$ denote the
family of all orbits of the action of $N \times K$ under
$\mathcal{A}_q$. Suppose that each orbit of the action of
$S_{3}\times K$ on $\mathcal{A}_q$ contains an element $u$ which can
be written as $u=\lambda h+ \mu e_{j}$ for some $h\in \mathcal{L}$,
$\lambda $, $\mu \in \mathbb{F}_{q}$ and $j\in
\{1,2,3\}$. Thus the set  $\mathcal{L} \cup \{(1,1,1)\}$ is a short covering of $\mathbb{F}%
_{q}^{3}.$
\end{theorem}

\begin{proof}
Take an arbitrary vector $v$ in $\mathbb{F}_{q}^{3}$. We analyze two
cases.

Case 1: the case where $v\in \mathbb{F}_{q}^{3}\setminus
\mathcal{A}_q$. Here $v$ has at least two coincident coordinates,
say $v=(\lambda,\lambda,\mu )$. Thus  $v=\lambda
(1,1,1)+(\mu-\lambda) e_{3}$, that is, $v \in E((1,1,1))$.

Case 2: $v\in \mathcal{A}_q$. We show that $\cup _{h\in
\mathcal{L}}{E}(h)$ contains $v$. Since $v\in \mathcal{A}_q$, there
is a vector $u\in \mathcal{O}$ such that $u=v^{\sigma \gamma }$. By
hypothesis, $u=\lambda h+\mu e_{j}$, where $\lambda \neq 0$. By
applying $\sigma $, we obtain $u^{\sigma }=\lambda' h+\mu' e_{j}$.
By applying $\varphi $, we have $v=u^{\sigma \varphi}=\lambda'
h^{\varphi }+\mu' e_{j}^{\varphi }$. Because $\mathcal{L}$ is an
invariant set under the action of $\varphi $, the required statement
is obtained.
\end{proof}

\begin{example}
\emph{The theorem above is optimal for small instances $q=3,4,5,7$.
For example, the sharp bound $c(5)\leq 4$ can be reproved as
follows. Choose $\mathcal{L}=\{(0,2,3),(3,0,2),(2,3,0)\}$. Because
$\mathcal{L}$ is invariant under the action of the 3-cycle $\varphi:
(u_1,u_2,u_3) \mapsto (u_2,u_3,u_1),$ take $N= <\varphi>$ the
subgroup generated by $\varphi$. The action of $G=<\varphi> \times
K$ on $\mathcal{A}_5$ produces five orbits. Since the stabilizer of
a vector $u$ is the trivial subgroup, each orbit $u^{G}$ has twelve
elements. Moreover, the representatives are covered by
$\mathcal{L}$, as described below
\[
\begin{array}{ll}
(0,1,2)=3(0,2,3)+3e_{3}  & (0,1,3)=3(0,2,3)+4e_{3} \\
(0,1,4)=3(0,2,3)         & (1,2,3)=1(0,2,3)+1e_{1} \\
(1,3,2)=4(0,2,3)+1e_{1}.&
\end{array}
\]
The bound follows from Theorem \ref{metodo}.}
\end{example}

\begin{proposition}\label{superior8}
The upper bound $c(8) \leq 6$ holds.
\end{proposition}
\begin{proof}
Choose the vectors
\[
\begin{array}{lll}
h_1= (1,1,1) ,            & h_2= (0,0,\xi^1)    ,   & h_3=(1,\xi^1,0) ,    \\
h_4= (1,\xi^2,\xi^3) ,    & h_5=(1,\xi^3,\xi^2) ,   &
h_6=(\xi^6,\xi^5,1),
\end{array}
\]
where $\xi$ denotes a generator of the group $\mathbb{F}_8^{*}$. We
claim that $ \mathcal{H}=\{ h_1, \ldots, h_6 \}$ is a short covering
of $\mathbb{F}_8^3$. In order to apply Theorem \ref{metodo}, choose
the group $G$ isomorphic to $K$.

Take an arbitrary vector $u=(u_1,u_2,u_3) \in \mathbb{F}_8^3$. We
analyze now the case where $u=(u_1,u_2,u_3)  \in \mathbb{F}_8^3
\setminus \mathcal{D}_8$.  If there are at least two coincident
coordinates in $u$, then $u$ is covered by $h_1$. If zero appears as
a coordinate of $u$, then $u$ is covered (shortly) by $h_2$ or
$h_3$.

Otherwise,  $u$ belongs to $\mathcal{D}_8$. The representatives of
the orbits on $\mathcal{D}_8$ can be chosen as $u=(1,u_2,u_3)$,
where $u_2,u_3\in \{ \xi^1,\xi^2,\xi^3,\xi^4,\xi^5,\xi^6\}$. For the
cases where $u_2\in \{\xi^1,\xi^2,\xi^3 \}$ or $u_3 \in
\{\xi^2,\xi^3 \}$, a simple look shows us that $u$ is covered
(shortly) by the vectors $h_3$, $h_4$ or $h_5$. It remains to
analyze the cases $u_2 \not\in \{\xi^1,\xi^2,\xi^3 \}$, and $u_3
\not\in \{\xi^2,\xi^3 \}$, that is, we still need to show that
$\mathcal{H}$ is a short covering of the representatives
\[
(1,\xi^4,\xi^1), \ (1,\xi^4,\xi^5), \ (1,\xi^4,\xi^6), \
(1,\xi^5,\xi^1), \ (1,\xi^5,\xi^4), \ (1,\xi^5,\xi^6), \
(1,\xi^6,\xi^1), \ (1,\xi^6,\xi^4), \ (1,\xi^6,\xi^5).
\]
From the equalities
\[
\xi^2(1,\xi^2,\xi^3) = (\xi^2,\xi^4,\xi^5), \ \mbox{ and } \
\xi^3(1,\xi^2,\xi^3) = (\xi^3,\xi^5,\xi^6),
\]
the vector $h_4=(1,\xi^2,\xi^3)$ covers $(1,\xi^4,\xi^5)$, and
$(1,\xi^5,\xi^6)$. The equality
$$
\xi^2(1,\xi^3,\xi^2) = (\xi^2,\xi^5,\xi^4)
$$
implies that $h_5=(1,\xi^3,\xi^2)$ covers $(1,\xi^5,\xi^4)$. From
\[
\xi^1(\xi^6,\xi^5,1) = (1,\xi^6,\xi^1), \mbox{ and } \
\xi^6(\xi^6,\xi^5,1) =  (\xi^5,\xi^4,\xi^6),
\]
the vector $h_6= (\xi^6,\xi^5,1)$ covers $(1,\xi^4,\xi^1)$,
$(1,\xi^6,\xi^1)$, $(1,\xi^6,\xi^4)$, $(1,\xi^6,\xi^5)$,
$(1,\xi^5,\xi^1)$, and $(1,\xi^4,\xi^6)$.

Therefore  $\mathcal{H}$ is a short covering of $\mathbb{F}_8^3$.
\end{proof}

\begin{proposition}\label{superior9}
We obtain $c(9) \leq 6$.
\end{proposition}

\begin{proof}
Consider the vectors
\[
\begin{array}{llllll}
h_1=(1,1,1),          & h_2=(1,0,0),         & h_3=(0,1,\xi^4),\\
h_4= (1,\xi^2,\xi^4), & h_5=(1,\xi^4,\xi^2), & h_6(1,\xi^6,\xi^6),
\end{array}
\]
where $\xi$ denotes an generator of $\mathbb{F}_9^{*}$. Let us show
that $\mathcal{H}=\{ h_1,\ldots,h_6\}$ is a short covering of
$\mathbb{F}_9^3$.

Let $u$ be an arbitrary vector in $\mathbb{F}_9^3$. If $u \in
\mathbb{F}_9^3 \setminus \mathcal{D}_9$, then $u$ has at least two
coincident coordinates, or zero appears at least in one coordinate
of $u$. In the first case, $u$ is covered by $(1,1,1)$, and $u$ is
covered by $(1,0,0)$ and $(0,1,\xi^4)$ in the second.

We suppose now $u=(u_1,u_2,u_3)\in \mathcal{D}_9$. Assume that
$u=(1,u_2,u_3)$, where $u_2,u_3\in \{
\xi^1,\xi^2,\xi^3,\xi^4,\xi^5,\xi^6,\xi^7\}$. If $u_2\in
\{\xi^2,\xi^4,\xi^6 \}$ or $u_3 \in \{\xi^2,\xi^4,\xi^6 \}$, then
$u$ is covered by the vectors $(1,\xi^2,\xi^4)$, $(1,\xi^4,\xi^2)$
or $(1,\xi^6,\xi^6)$.

We need to show that $\mathcal{H}$ is a short covering of the
vectors below:
\begin{equation}\label{bloco9}
\begin{array}{llllllllllllll}
  (1,\xi^1,\xi^3), & (1,\xi^1,\xi^5), & (1,\xi^1,\xi^7), &
  (1,\xi^3,\xi^1), & (1,\xi^3,\xi^5), & (1,\xi^3,\xi^7), \\
  (1,\xi^5,\xi^1), & (1,\xi^5,\xi^3), & (1,\xi^5,\xi^7), &
  (1,\xi^7,\xi^1), & (1,\xi^7,\xi^3), & (1,\xi^7,\xi^5).
\end{array}
\end{equation}
Note that
\[
\begin{array}{lllllllllllllll}
\xi^1(0,1,\xi^4) = (0,\xi^1,\xi^5), & \xi^3(0,1,\xi^4)    =
(0,\xi^3,\xi^7) \\
 \xi^5(0,1,\xi^4)=
(0,\xi^5,\xi^1), & \xi^7(0,1,\xi^4) = (0,\xi^7,\xi^3), &
\end{array}
  \]
thus $h_3=(0,1,\xi^4)$ covers $(1,\xi^1,\xi^5)$,
  $(1,\xi^3,\xi^7)$, $(1,\xi^5,\xi^1)$, and $(1, \xi^7, \xi^3)$.
The equalities
\[
\begin{array}{llll}
  \xi^1(1,\xi^2,\xi^4) = (\xi^1,\xi^3,\xi^5),   & \xi^3(1,\xi^2,\xi^4)  =(\xi^3,\xi^5,\xi^7), \\
  \xi^5(1,\xi^2,\xi^4)  =  (\xi^5,\xi^7,\xi^1), & \xi^7(1,\xi^2,\xi^4)      = (\xi^7,\xi^1,\xi^3)
\end{array}
\]
imply that $h_4=(1,\xi^2,\xi^4)$ covers $(1,\xi^3,\xi^5)$,
$(1,\xi^5,\xi^7)$, $(1,\xi^7,\xi^1)$, and $(1,\xi^1,\xi^3)$. Since
\[
\begin{array}{llllllllll}
 \xi^1(1,\xi^4,\xi^2) = (\xi^1,\xi^5,\xi^3), & \xi^3(1,\xi^4,\xi^2)
  =(\xi^3,\xi^7,\xi^5),\\
  \xi^5(1,\xi^4,\xi^2)=(\xi^5,\xi^1,\xi^7)  , & \xi^7(1,\xi^4,\xi^2)  = (\xi^7,\xi^3,\xi^1),
\end{array}
\]
the vector $h_5=(1,\xi^4,\xi^2)$ covers $(1,\xi^5,\xi^3)$,
$(1,\xi^7,\xi^5)$, $(1,\xi^1,\xi^7)$ and $(1,\xi^3,\xi^1)$.

Because the vectors in (\ref{bloco9}) are covered by $(0,1,\xi^4)$,
$(1,\xi^2,\xi^4)$, or $(1,\xi^4,\xi^2),$ we conclude that
$\mathcal{H}$ is a short covering of $\mathbb{F}_{9}^3$.
\end{proof}

\section{Conclusion}

We conclude this work with the following contribution to the
computation of the function $c$.

\begin{proof}[Proof of Theorem \ref{anderson}]
It is an immediately consequence of the results in this work. In
fact, the lower bound $c(7) \geq 5$ follows from Proposition
\ref{inferior7}, while the upper bound $c(7) \leq 5$ follows from
the inequality (\ref{bestinfsup}). The value $c(8)=6$ is consequence
of the Propositions \ref{inferior8}, and \ref{superior8}. We obtain
from Propositions \ref{inferior9}, and \ref{superior9} that
$c(9)=6$.
\end{proof}

\end{document}